\documentclass[10pt,a4paper]{amsart}
\usepackage[a4paper,marginratio=1:1]{geometry}
\usepackage[protrusion=all]{microtype}
\usepackage{amssymb}
\usepackage[initials,non-sorted-cites]{amsrefs}
\usepackage{mybibspec}
\usepackage[shortlabels]{enumitem}
\usepackage[symbol,perpage]{footmisc}
\usepackage{graphicx}
\usepackage{braket,tensor,bm,stmaryrd,fouridx}

\numberwithin{equation}{section}
\usepackage{mathtools}\mathtoolsset{showonlyrefs}

\newtheorem{thm}{Theorem}[section]
\newtheorem{prop}[thm]{Proposition}
\newtheorem{lem}[thm]{Lemma}
\newtheorem{cor}[thm]{Corollary}
\theoremstyle{definition}
\newtheorem{dfn}[thm]{Definition}
\theoremstyle{remark}
\newtheorem{rem}[thm]{Remark}

\newcommand{\abs}[1]{\lvert#1\rvert}

\newcommand{\conj}[1]{{\overline{#1}}}
\newcommand{\norm}[1]{\lVert#1\rVert}
\newcommand{\bdry}{\partial}
\newcommand{\orddot}{\mathord{\cdot}}
\newcommand{\compose}{\mathbin{\circ}}
\newcommand{\transpose}[1]{{\fourIdx{t}{}{}{}{#1}}}
\DeclareMathOperator{\GL}{GL}
\DeclareMathOperator{\PU}{PU}
\DeclareMathOperator{\PO}{PO}
\DeclareMathOperator{\Aut}{Aut}
\DeclareMathOperator{\Res}{Res}
\DeclareMathOperator{\Length}{Length}
\newcommand{\ch}{\smash{\operatorname{ch}}_\mathcal{H}}
\newcommand{\chtilde}{\smash{\widetilde{\operatorname{ch}}}_\mathcal{H}}
\newcommand{\dCC}{d_{\mathrm{CC}}}
\newcommand{\dK}{d_{\mathcal{H}}}
\newcommand{\BK}{\mathbb{B}_{\mathcal{H}}}
\newcommand{\normH}[1]{\norm{#1}_{\mathcal{H}}}
\newcommand{\comptwoform}{\Omega}
\newcommand{\comptwoformH}{\Omega_{\mathcal{H}}}

\title[CR-invariant energy of Legendrian knots in the Heisenberg group]{CR-invariant energy of Legendrian knots\\ in the Heisenberg group}
\author[Y. Matsumoto]{Yoshihiko Matsumoto}
\address{Department of Mathematics, Graduate School of Science, The University of Osaka, Toyonaka, Osaka 560-0043, Japan}
\email{matsumoto.yoshihiko.sci@osaka-u.ac.jp}
\author[J. O'Hara]{Jun O'Hara}
\address{Department of Mathematics and Informatics, Faculty of Science, Chiba University, Inage, Chiba 263-8522, Japan}
\email{ohara@math.s.chiba-u.ac.jp}

\begin{document}

\begin{abstract}
	We introduce an energy functional for Legendrian knots in the 3-dimensional Heisenberg group $\mathcal{H}$,
	which serves as a sub-Riemannian analog of
	the M\"obius invariant knot energy in Euclidean 3-space introduced by the second author.
	The energy is obtained by regularizing a divergent integral of the potential of order $-2$
	with respect to the Kor\'anyi distance on $\mathcal{H}$;
	this choice of distance is essential for the energy to be invariant under the action of $\PU(2,1)$.
	We characterize $\mathbb{R}$-circles in $\mathcal{H}$ as the minimizers of the energy,
	and establish a Heisenberg analog of the Doyle--Schramm cosine formula.
	We also show that the energy integrand admits an expression in terms of
	a complex-valued 2-form on the complement of the diagonal in $\mathcal{H}\times\mathcal{H}$,
	providing a partial analog of the infinitesimal cross ratio interpretation known from the classical setting.
\end{abstract}

\maketitle

\section*{Introduction}

Knot energies in Euclidean 3-space were introduced by the second author \cite{OHara-91}
in response to a problem posed by Fukuhara and Sakuma: to define an energy functional on the space of knots
so that a canonical representative may be obtained as a minimizer within each knot type.
These energies are defined by regularizing divergent integrals of $r^{-\alpha}$-potentials.
Among them, the energy for $\alpha=2$ is known to be distinguished:
it is the smallest $\alpha$ for which the corresponding regularized integral diverges to $+\infty$
as the knot approaches a self-intersection.

Freedman, He, and Wang \cite{Freedman-He-Wang-94} showed that the $r^{-2}$-energy
is invariant under M\"obius transformations of $\mathbb{R}^3$,
and that each prime knot type admits an energy minimizer.
This initiated the study of geometric knot theory from the viewpoint of energy functionals,
and the theory has since been extended to higher-dimensional submanifolds and to Riemannian manifolds.

In the present work, we consider the sub-Riemannian case of the 3-dimensional Heisenberg group.
We define an analogous energy functional for Legendrian (i.e., horizontal) knots
in the Heisenberg group $\mathcal{H}=\mathbb{C}\times\mathbb{R}$ and investigate its properties.
Along the lines of the classical construction, we regularize the
divergent integral of the $r^{-2}$-potential
with respect to the Kor\'anyi distance on $\mathcal{H}$, which is denoted by $\dK$ in this paper,
to define our energy.
The use of the Kor\'anyi distance $\dK$,
rather than the (equally, if not more, well-known) Carnot--Carath\'eodory distance,
is essential:
it is this choice that makes the integrand $dp\,dq/\dK(p,q)^2$ on $(K\times K)\setminus\Delta$ invariant
under the diagonal action of $\PU(2,1)$,
where $\PU(2,1)$ is the group of CR automorphisms of the one-point compactification $\mathcal{H}\cup\set{\infty}$
of the Heisenberg group.

This group of CR automorphisms is the counterpart in our setting of
the group $\PO(4,1)$ of M\"obius transformations in the classical theory,
and as in that case,
we are able to show that our energy is invariant under the $\PU(2,1)$-action.
For convenience, borrowing the term from the classical setting,
we refer to elements of $\PU(2,1)$ as \emph{M\"obius transformations} on $\mathcal{H}\cup\set{\infty}$,
and accordingly say that our energy is \emph{M\"obius invariant}.
Just as in the classical case,
we emphasize that the M\"obius invariance of our energy is
not an immediate consequence of the invariance of the integrand alone.
The energy is defined via a regularization process that is not M\"obius invariant,
and it is a priori unclear whether the resulting finite quantity is M\"obius invariant.

After defining the energy,
we prove that $\mathbb{R}$-circles, a distinguished class of Legendrian (un)knots in $\mathcal{H}$,
are the unique minimizers, playing the same role as circles do in the classical theory.
The Heisenberg version of the cosine formula,
which is due to Doyle and Schramm in the classical setting,
is also established by replacing circles with $\mathbb{R}$-circles.

Furthermore, recall from Langevin--O'Hara \cite{Langevin-OHara-05} that
the integrand of the energy in the Euclidean setting admits an interpretation
in terms of a complex-valued 2-form $\comptwoform$ on $(\mathbb{R}^3\times\mathbb{R}^3)\setminus\Delta$
known as the infinitesimal cross ratio \cite{Langevin-OHara-05}.
We establish a partial analog of this interpretation in the Heisenberg setting:
the energy integrand $dp\,dq/\dK(p,q)^2$ can be expressed as the absolute value
of the pullback of a complex-valued 2-form $\comptwoformH$ on $(\mathcal{H}\times\mathcal{H})\setminus\Delta$.
However, the cross-ratio construction does not just extend to $\comptwoformH$;
it requires a more refined approach adapted to the CR-geometric setting.

This paper is organized in the following manner.
Some basic facts and our conventions regarding the Heisenberg group $\mathcal{H}$ are
reviewed in \S\ref{sec:Heisenberg-space}.
In \S\ref{sec:beta-function},
we introduce a Heisenberg analog of Brylinski's beta function \cite{Brylinski-99},
which allows us to define the energy of Legendrian knots as a regularization of a divergent integral.
The M\"obius invariance of the energy is established in \S\ref{sec:Mobius-invariance},
and in \S\ref{sec:R-circles} we characterize $\mathbb{R}$-circles
as those Legendrian knots with least possible energy (namely zero).
After that, the cosine formula is proved in \S\ref{sec:cosine-formula},
and the final section is devoted to further discussion of the energy integrand $dp\,dq/\dK(p,q)^2$,
including the interpretation mentioned in the last paragraph.

The first and second authors are partially supported by JSPS KAKENHI Grant Numbers 24K06738 and 23K03083,
respectively.

\section{The Heisenberg group}
\label{sec:Heisenberg-space}

We recall standard definitions concerning the (3-dimensional) Heisenberg group
$\mathcal{H}=\mathbb{C}\times\mathbb{R}$.
Our main references are Capogna--Danielli--Pauls--Tyson \cite{Capogna-Danielli-Pauls-Tyson-07}
and Folland \cite{Folland-05}.
The standard coordinates of $\mathcal{H}$ will be denoted by $(z,u)$, or $(x,y,u)$, where $z=x+iy$;
we reserve the letter $t$, which is commonly used for the last coordinate, for other purposes.

To be specific, we use the ``exponential coordinates'' on $\mathcal{H}$
(see the ``Notes on Notation'' in \cite{Folland-05}),
which means that the group structure on $\mathcal{H}$ is defined by
\begin{equation}
	\label{eq:group-structure-complex}
	(z,u)\cdot(z',u')=\left(z+z',u+u'-\frac{1}{2}\Im(z\conj{z}')\right)
\end{equation}
or
\begin{equation}
	\label{eq:group-structure}
	(x,y,u)\cdot(x',y',u')=\left(x+x',y+y',u+u'+\frac{1}{2}(xy'-x'y)\right).
\end{equation}
The space $\mathcal{H}$ is equipped with the left-invariant contact distribution $H$
spanned by
\begin{equation}
	\label{eq:horizontal-vector-fields}
	X=\frac{\partial}{\partial x}-\frac{1}{2}y\frac{\partial}{\partial u},\qquad
	Y=\frac{\partial}{\partial y}+\frac{1}{2}x\frac{\partial}{\partial u},
\end{equation}
and moreover, the standard CR structure on $\mathcal{H}$ is defined by the complex span of
\begin{equation}
	\label{eq:complex-horizontal-vector-field}
	Z=\frac{1}{2}(X-iY)=\frac{\partial}{\partial z}-\frac{i}{4}\conj{z}\frac{\partial}{\partial u}.
\end{equation}
Note also that
\begin{equation}
	\label{eq:contact-form}
	\theta=du+\frac{1}{2}(y\,dx-x\,dy)
	=du+\frac{1}{2}\Im(z\,d\conj{z})
\end{equation}
is a left-invariant contact 1-form annihilating $H$.

For $\lambda>0$, the Heisenberg \emph{dilation}
$\delta_\lambda\colon\mathcal{H}\to\mathcal{H}$ is defined by
\begin{equation}
	\delta_\lambda\colon(x,y,u)\mapsto(\lambda x,\lambda y,\lambda^2u).
\end{equation}
Clearly, the group structure and the contact distribution $H$ are both invariant under the dilation,
whereas $(\delta_\lambda)_*X_p=\lambda X_{\delta_\lambda(p)}$
and $(\delta_\lambda)_*Y_p=\lambda Y_{\delta_\lambda(p)}$.

The inner product on $H$, which we write $\braket{\orddot,\orddot}$, is defined so that
$X_p$ and $Y_p$ form an orthonormal basis at each $p\in\mathcal{H}$.
It follows that the inner product is left-invariant and homogeneous of degree $2$ for dilations.

A smooth path $\gamma\colon[a,b]\to\mathcal{H}$ is \emph{Legendrian} (or \emph{horizontal})
if $\dot{\gamma}(t)\in H_{\gamma(t)}$ for all $t\in [a,b]$,
where $H_p$ is the horizontal plane through $p$.
Its length is defined by
\begin{equation}
	\label{eq:length-of-a-path}
	L(\gamma)=\int_a^b\abs{\dot{\gamma}}\,dt,
\end{equation}
where $\abs{\orddot}$ is the norm with respect to $\braket{\orddot,\orddot}$.
These definitions are generalized to piecewise smooth paths as usual.
In this regard, note that the \emph{vertical projection}
$\pi\colon\mathcal{H}\to\mathbb{C}$ given by $(z,u)\mapsto z$ is a convenient tool:
the length of a Legendrian path $\gamma$ is nothing but
the Euclidean length of its projection $\pi\compose\gamma$.
This indicates that
measurements in $\mathcal{H}$ is, in some sense, simpler than those in $\mathbb{R}^3$.

An easy, but important, observation is that
a Legendrian path $\gamma\colon I\to\mathcal{H}$ is determined by
its projection $\pi\compose\gamma$ and an initial point $p=\gamma(t_0)$, where $t_0\in I$.
This is because $t\mapsto(x(t),y(t),u(t))$ is Legendrian if and only if
$\dot{u}+(1/2)(y\dot{x}-x\dot{y})=0$.
To be explicit, the Legendrian lift of the path $t\mapsto(x(t),y(t))$
with prescribed initial point $p=(x(t_0),y(t_0),u(t_0))$ is uniquely given by
\begin{equation}
	\label{eq:u-coordinate-change-along-Legendrian-path}
	u(t)=u(t_0)-\frac{1}{2}\int_{t_0}^t(y(\tau)\dot{x}(\tau)-x(\tau)\dot{y}(\tau))\,d\tau.
\end{equation}

There are two different commonly used left-invariant distance functions on $\mathcal{H}$.
One of them is the \emph{Carnot--Carath\'eodory distance}, which is defined by
\begin{equation}
	\label{eq:Carnot-Caratheodory-distance}
	\dCC(p,q)=\inf_\gamma L(\gamma),
\end{equation}
where $\gamma$ ranges over all piecewise smooth paths from $p$ to $q$.
In this paper, however, we exclusively use the other one, which is known as the \emph{Kor\'anyi distance}:
\begin{equation}
	\dK(p,q)=\norm{p^{-1}q}_{\mathcal{H}},
	\qquad\text{where}\qquad
	\normH{(x,y,u)}=\sqrt[4]{(x^2+y^2)^2+16u^2}.
\end{equation}
Since it is the only distance function on $\mathcal{H}$ used in this paper, we simply denote it by $\dK$.

\begin{lem}
	\label{lem:infinitesimal-Koranyi-distance}
	If $\gamma$ is Legendrian, then
	\begin{equation}
		\left.\frac{d}{dt}\dK(\gamma(t_0),\gamma(t))\right|_{t=t_0}=\abs{\dot{\gamma}(t_0)}.
	\end{equation}
\end{lem}

\begin{proof}
	We may assume $t_0=0$.
	Moreover, since $\dK$ and $\abs{\orddot}$ are left-invariant, we may assume without losing generality
	that $\gamma(t_0)=(0,0,0)$.
	Then, \eqref{eq:u-coordinate-change-along-Legendrian-path} implies that $u(t)=O(t^3)$,
	and hence
	\begin{equation}
		(x(t)^2+y(t)^2)^2+16u(t)^2=t^4(\dot{x}(0)^2+\dot{y}(0)^2)^2+O(t^5).
	\end{equation}
	Consequently,
	\begin{equation}
		\left.\frac{d}{dt}\dK(\gamma(0),\gamma(t))\right|_{t=0}
		=\left.\frac{d}{dt}\normH{\gamma(t)}\right|_{t=0}
		=\sqrt{\dot{x}(0)^2+\dot{y}(0)^2}
		=\abs{\dot{\gamma}(0)}.\qedhere
	\end{equation}
\end{proof}

\begin{cor}
	\label{cor:length-in-terms-of-Koranyi-distance}
	The length of a Legendrian path $\gamma\colon I\to\mathcal{H}$
	defined by the Kor\'anyi distance coincides with \eqref{eq:length-of-a-path}, i.e.,
	\begin{equation}
		L(\gamma)=\inf\sum_i\dK(\gamma(t_i),\gamma(t_{i+1})),
	\end{equation}
	where the infimum is taken over all partitions of $I$.
\end{cor}

The crucial advantage of the Kor\'anyi distance over the Carnot--Carath\'eodory distance
is its simple behavior under the Heisenberg \emph{inversion}
\begin{equation}
	\label{eq:inversion}
	\iota\colon
	\mathcal{H}\cup\set{\infty}\to\mathcal{H}\cup\set{\infty},\qquad
	(z,u)\mapsto\left(\frac{z}{\abs{z}^2+4iu},-\frac{u}{\abs{z}^4+16u^2}\right).
\end{equation}
Namely, for $p$, $q\in\mathcal{H}\setminus\set{(0,0,0)}$,
\begin{equation}
	\label{eq:inversion-formula-for-Koranyi-distance}
	\dK(\iota(p),\iota(q))=\frac{\dK(p,q)}{\norm{p}_{\mathcal{H}}\norm{q}_{\mathcal{H}}}.
\end{equation}
By Lemma \ref{lem:infinitesimal-Koranyi-distance}, we can deduce the infinitesimal version
of \eqref{eq:inversion-formula-for-Koranyi-distance}:
if $\gamma$ is a Legendrian path not passing through the origin in $\mathcal{H}$,
then $\tilde{\gamma}=\iota\compose\gamma$ satisfies
\begin{equation}
	\label{eq:inversion-formula-for-horizontal-vector-norm}
	\abs{\dot{\tilde{\gamma}}(t)}=\frac{\abs{\dot{\gamma}(t)}}{\norm{\gamma(t)}_\mathcal{H}^2}.
\end{equation}

Next, we recall how we can regard the standard 3-sphere $S^3$,
which is the boundary of the unit ball $B^2_{\mathbb{C}}$ in $\mathbb{C}^2$,
as the one-point compactification of $\mathcal{H}$.
First, the complex unit ball $B^2_{\mathbb{C}}$ is biholomorphic to the \emph{Siegel domain}
\begin{equation}
	\label{eq:Siegel-domain}
	D=\set{(z,w)\in\mathbb{C}^2|\Im w>\abs{z}^2}
\end{equation}
via the \emph{Cayley transform}
\begin{equation}
	\label{eq:Cayley}
	\Phi_C\colon B^2_{\mathbb{C}}\to D,\qquad
	(\zeta^1,\zeta^2)\mapsto\left(\frac{\zeta^1}{i(\zeta^2-1)},\frac{\zeta^2+1}{i(\zeta^2-1)}\right),
\end{equation}
which extends continuously to the CR diffeomorphism $S^3\setminus\set{(0,1)}\cong\bdry D$.
We then identify $\mathcal{H}$ and $\bdry D$ by
\begin{equation}
	\label{eq:identification-of-Heisenberg-and-Siegel-boundary}
	\mathcal{H}\to\bdry D,\qquad
	(z,u)\mapsto (z,-4u+i\abs{z}^2).
\end{equation}
By introducing the point $\infty$ corresponding to $(0,1)\in S^3$,
we obtain an identification of $S^3$ and $\mathcal{H}\cup\set{\infty}$.

A consequence from the above construction is that the one-point compactification $\mathcal{H}\cup\set{\infty}$
carries a CR structure induced by the complex structure of $\mathbb{C}^2$
(the CR structure we introduced earlier on $\mathcal{H}$ is nothing but the restriction of this one).
Therefore, it is natural to seek invariant quantities for Legendrian paths or knots
in $S^3=\mathcal{H}\cup\set{\infty}$ under the action of CR automorphisms.
To describe these automorphisms,
we embed the complex 2-space $\mathbb{C}^2$ into $\mathbb{C}P^2$ by
$(z,w)\mapsto[1:z:w]$
and consider the projective actions of matrices in $\GL(3,\mathbb{C})$, i.e.,
complex fractional linear transformations.
It is well known that the group of CR automorphisms of $S^3$ consists precisely of those transformations
given by the matrices preserving the indefinite Hermitian form
\begin{equation}
	I_{B^2_{\mathbb{C}}}=
	\begin{pmatrix}
		-1 & 0 & 0 \\
		0 & 1 & 0 \\
		0 & 0 & 1
	\end{pmatrix}.
\end{equation}

For computational purposes, it is more convenient to work in the Heisenberg coordinates
rather than regarding $\mathcal{H}\cup\set{\infty}$ as a 3-sphere.
Since the Cayley transform \eqref{eq:Cayley} is a fractional linear transformation
given by the matrix
\begin{equation}
	C=
	\begin{pmatrix}
		-i & 0 & i \\
		0 & 1 & 0 \\
		1 & 0 & 1
	\end{pmatrix},
\end{equation}
the indefinite Hermitian form given by
\begin{equation}
	I_D=\smash{\transpose{C}}^{-1}I_{B^2_{\mathbb{C}}}\smash{\conj{C}}^{-1}=
	\begin{pmatrix}
		0 & 0 & -i/2 \\
		0 & 1 & 0 \\
		i/2 & 0 & 0
	\end{pmatrix}
\end{equation}
corresponds to the CR automorphisms of $\bdry D\cup\set{\infty}$.
We use the notation $\PU(2,1)$ for the group of fractional linear transformations that respect $I_D$.
By observing the identification \eqref{eq:identification-of-Heisenberg-and-Siegel-boundary},
we regard this group as the CR automorphism group of $\mathcal{H}\cup\set{\infty}$:
\begin{equation}
	\PU(2,1)=\Aut_{\mathrm{CR}}(\mathcal{H}\cup\set{\infty}).
\end{equation}
We simply call its elements
as \emph{M\"obius transformations} on $\mathcal{H}\cup\set{\infty}$,
or on $\mathcal{H}$ when there is no fear of confusion,
in analogy with the classical M\"obius transformations acting on Euclidean space.

It is not so tedious to check that
the group $\PU(2,1)$ is generated by the dilations $\delta_\lambda$,
the Heisenberg translations (the left multiplications in $\mathcal{H}$),
the rotations $R_\theta\colon(z,u)\mapsto(e^{i\theta}z,u)$,
and the inversion $\iota$ defined by \eqref{eq:inversion}.

\section{The beta function associated with Legendrian knots}
\label{sec:beta-function}

A \emph{Legendrian knot} in $\mathcal{H}$ is a closed Legendrian path without self-intersections.
We denote it by $K$ when regarded as a subset of $\mathcal{H}$,
and by $\gamma\colon I\to\mathcal{H}$ when regarded as a parametrized curve
(we understand that $I$ is a bounded closed interval).
We always consider smooth knots,
and unless otherwise stated, we assume that they are parametrized by arc-length $s$.
It is convenient to extend $\gamma\colon I\to\mathcal{H}$ to
a smooth periodic mapping $\gamma\colon\mathbb{R}\to\mathcal{H}$ with period $L=L(\gamma)$.

For such a Legendrian knot $K$,
\eqref{eq:inversion-formula-for-Koranyi-distance} and \eqref{eq:inversion-formula-for-horizontal-vector-norm}
imply that the 2-form on $(K\times K)\setminus\Delta$
(where $\Delta$ denotes the diagonal) given by
\begin{equation}
	\label{eq:invariant-2-form-on-KxK}
	\frac{dp\,dq}{\dK(p,q)^2},
\end{equation}
by which we mean
\begin{equation}
	\frac{ds\wedge ds'}{\dK(\gamma(s),\gamma(s'))^2},
\end{equation}
is invariant under the diagonal action of $\PU(2,1)$.
Therefore, one might expect that we can obtain a M\"obius invariant quantity
by integrating \eqref{eq:invariant-2-form-on-KxK} over $K\times K$,
but the singularity of \eqref{eq:invariant-2-form-on-KxK} along the diagonal
makes the integral diverge, and regularization is required.
In this section, we carry out this regularization
following Brylinski's method \cite{Brylinski-99} in the classical Euclidean theory.
Other methods will be discussed in the next section.

We define the \emph{beta function} of a Legendrian knot $K$ by
\begin{equation}
	B_K(\zeta)=\iint_{K\times K}\dK(p,q)^\zeta\,dp\,dq
\end{equation}
for $\Re\zeta>-1$ (so that the integral converges).
Alternatively, let $\Psi\colon[0,\infty)\to[0,\infty)$ be
the cumulative distribution function of the interpoint distance
\begin{equation}
	\Psi(t)=\mu_{K\times K}(\set{(p,q)\in K\times K|\dK(p,q)<t}),
\end{equation}
$\mu_{K\times K}$ being the product measure induced by the arc-length measure $\mu_K$ on $K$.
Then we can express $B_K(\zeta)$ as
\begin{equation}
	\label{eq:Gamma-function-as-Stieltjes-integral}
	B_K(\zeta)=\int_0^\infty t^\zeta\,d\Psi(t)
\end{equation}
as a Riemann--Stieltjes integral.

We also define
\begin{equation}
	\Psi^{\mathrm{loc}}_p(t)
	=\Length(K\cap\BK(p;t))
	=\mu_K(\set{q\in K|\dK(p,q)<t})
\end{equation}
for any fixed $p\in K$ and $t>0$,
where $\BK(p;t)$ denotes the Kor\'anyi ball. Then,
\begin{equation}
	\label{eq:CDF-as-integral}
	\Psi(t)=\int_K\Psi^\mathrm{loc}_p(t)\,dp.
\end{equation}
The following result for $\Psi^\mathrm{loc}_p$,
which is an analogue of Proposition 3.1 (i) in O'Hara--Solanes \cite{OHara-Solanes-18},
is crucial for the subsequent discussion of the meromorphic continuation of $B_K$.

\begin{prop}
	\label{prop:smooth-extendability-of-local-interpoint-distance-distribution-function}
	Let $K$ be a Legendrian knot and $p\in K$.
	Then, for sufficiently small $\varepsilon>0$,
	$\Psi^\mathrm{loc}_p$ can be extended to a smooth odd function
	defined on the interval $(-\varepsilon,\varepsilon)$.
	Furthermore, $\varepsilon$ can be taken independently of $p$.
\end{prop}

To show this, let us parametrize the knot $K$ by $\gamma\colon I\to\mathcal{H}$ and let $p=\gamma(s_0)$.
After extending $\gamma$ periodically,
we define the \emph{signed Kor\'anyi chord length function} (or simply the \emph{signed chord length function})
$\ch\colon\mathbb{R}\to\mathbb{R}$ by
\begin{equation}
	\label{eq:signed-chord-length}
	\ch(s)=
	\begin{cases}
		\dK(p,\gamma(s)), &\qquad s>s_0,\\
		0,&\qquad s=s_0,\\
		-\dK(p,\gamma(s)), &\qquad s<s_0.
	\end{cases}
\end{equation}
Then, Proposition \ref{prop:smooth-extendability-of-local-interpoint-distance-distribution-function} follows
once we show the following lemma,
which means that $\ch(s)$ can also be used to parametrize $K$ near $p$.

\begin{lem}
	\label{lem:signed-Koranyi-chord-length-as-a-parametrization}
	In a sufficiently small neighborhood of $s_0\in\mathbb{R}$, $\ch$ is a smooth function satisfying $\ch'(s)>0$.
	Moreover, the size of such a neighborhood in $\mathbb{R}$ can be taken uniformly for all $p\in K$.
\end{lem}

Namely, once Lemma \ref{lem:signed-Koranyi-chord-length-as-a-parametrization} is established,
then the inverse function $\ch^{-1}$ is well-defined near $0\in\mathbb{R}$.
We then have, for sufficiently small $t>0$ (with the bound independent of $p$),
\begin{equation}
	\Psi^\mathrm{loc}_p(t)=\Length(K\cap\BK(p;t))=\ch^{-1}(t)-\ch^{-1}(-t),
\end{equation}
and hence $\Psi^\mathrm{loc}_p$ extends to the smooth odd function $t\mapsto\ch^{-1}(t)-\ch^{-1}(-t)$.

\begin{proof}[Proof of Lemma \ref{lem:signed-Koranyi-chord-length-as-a-parametrization}]
	By reparametrizing $K$ we may assume that $s_0=0$,
	and by left translations we may also assume that $p=\gamma(0)=(0,0,0)$.
	Let
	\begin{equation}
		\gamma(s)=(x(s),y(s),u(s)),
	\end{equation}
	by which we have
	\begin{equation}
		\dK(p,\gamma(s))=\normH{\gamma(s)}=\sqrt[4]{(x(s)^2+y(s)^2)^2+16u(s)^2}.
	\end{equation}
	Since $x(0)=y(0)=0$, we may write
	\begin{equation}
		x(s)=s\varphi(s),\qquad y(s)=s\psi(s),
	\end{equation}
	where $\varphi$ and $\psi$ are smooth functions satisfying $\varphi(0)^2+\psi(0)^2=1$.
	Furthermore, as in Lemma \ref{lem:infinitesimal-Koranyi-distance} we obtain $u(s)=O(s^3)$ as $s\to 0$,
	so we are able to write $u(s)=s^3\chi(s)$, where $\chi$ is a smooth function. Then,
	\begin{equation}
		\label{eq:Koranyi-distance-in-terms-of-arc-length}
		\dK(p,\gamma(s))
		=\sqrt[4]{s^4(\varphi(s)^2+\psi(s)^2)^2+16s^6\chi(s)^2}
		=\abs{s}\sqrt[4]{(\varphi(s)^2+\psi(s)^2)^2+16s^2\chi(s)^2}.
	\end{equation}
	Consequently, the signed chord length $\ch(s)$ is given by
	\begin{equation}
		\ch(s)=s\sqrt[4]{(\varphi(s)^2+\psi(s)^2)^2+16s^2\chi(s)^2},
	\end{equation}
	which is smooth and $\ch'(s)>0$ near $s=0$.
	The size of the neighborhood in which $\ch$ has these desired properties
	are determined by $\varphi(s)$, $\psi(s)$, $\chi(s)$, and their derivatives,
	which vary continuously in $p$.
	Therefore, the size of such a neighborhood can be chosen uniformly for all $p$.
\end{proof}

From Proposition \ref{prop:smooth-extendability-of-local-interpoint-distance-distribution-function}
and \eqref{eq:CDF-as-integral}, it follows that
$\Psi(t)$ also extends to a smooth odd function defined in $(-\varepsilon,\varepsilon)$.
Consequently, we can write
\begin{equation}
	B_K(\zeta)=\int_0^{\varepsilon/2}t^\zeta\Psi'(t)\,dt+(\text{an entire function}),
\end{equation}
and the conclusion below follows by standard arguments.

\begin{cor}
	By analytic continuation, $B_K(\zeta)$ extends to
	a meromorphic function defined on the whole complex plane,
	with at most simple poles at $\zeta=-1$, $-3$, $-5$, $\dotsc$.
\end{cor}

The final assertion of the corollary follows from
\begin{equation}
	\label{eq:residue-formula}
	\Res_{\zeta=-1-k}B_K(\zeta)=\frac{(-1)^k}{k!}\Psi^{(1+k)}(0),\qquad
	k=0,\,1,\,2,\,\dotsc.
\end{equation}

Explicit formulae for the residues at $\zeta=-1$, $-3$, $-5$, $\dotsc$ can be obtained
by following the steps given in the proof of
Lemma \ref{lem:signed-Koranyi-chord-length-as-a-parametrization} as detailed below.
Let $\overline{\gamma}=\pi\compose\gamma$ be the vertical projection of $\gamma$.
To compute $\Psi^\mathrm{loc}_p(t)$ for any fixed $p\in K$,
we may assume $\gamma(0)=p$ and $\overline{\gamma}'(0)=(1,0)$ without losing generality.
If $\kappa$ is the signed curvature of $\overline{\gamma}$, we have the asymptotic expansions
\begin{align}
	\label{eq:x-coord-of-legendrian-curves}
	x(s)
	&\sim s-\kappa^2s^3-\frac{\kappa\kappa'}{8}s^4
	+\frac{\kappa^4-4\kappa\kappa''-3(\kappa')^2}{120}s^5+\dotsb,\\
	\label{eq:y-coord-of-legendrian-curves}
	y(s)
	&\sim\frac{\kappa}{2}s^2+\frac{\kappa'}{6}s^3+\frac{-\kappa^3+\kappa''}{24}s^4+\dotsb
\end{align}
as $s\to 0$, where $\kappa(0)$, $\kappa'(0)$, ... are denoted simply by $\kappa$, $\kappa'$, .... Then we obtain
\begin{equation}
	\label{eq:u-coord-of-legendrian-curves}
	u(s)\sim\frac{\kappa}{12}s^3+\frac{\kappa'}{24}s^4+\dotsb
\end{equation}
by \eqref{eq:u-coordinate-change-along-Legendrian-path}.
Consequently,
\begin{equation}
	\label{eq:asymptotics-of-signed-chord-length}
	\ch(s)
	\sim
	s-\frac{\kappa^2}{72}s^3-\frac{\kappa\kappa'}{72}s^4
	+\frac{\kappa^4-72\kappa\kappa''-72(\kappa')^2}{17280}s^5+\dotsb
\end{equation}
and it follows that
\begin{equation}
	\label{eq:asymptotics-of-signed-chord-length-inverse}
	\ch^{-1}(t)\sim
	t+\frac{\kappa^2}{72}t^3+\frac{\kappa\kappa'}{72}t^4
	+\frac{\kappa^4+8\kappa\kappa''+8(\kappa')^2}{1920}t^5+\dotsb.
\end{equation}
Therefore,
\begin{equation}
	\Psi^\mathrm{loc}_p(t)
	=\ch^{-1}(t)-\ch^{-1}(-t)
	\sim
	2t+\frac{\kappa^2}{36}t^3+\frac{\kappa^4+8\kappa\kappa''+8(\kappa')^2}{960}t^5+\dotsb,
\end{equation}
from which we obtain
\begin{gather}
	\Res_{\zeta=-1}B_K(\zeta)=2L(\gamma),\qquad
	\Res_{\zeta=-3}B_K(\zeta)=\frac{1}{12}\int_K\kappa^2ds,\\
	\Res_{\zeta=-5}B_K(\zeta)=\int_K\frac{\kappa^4+8\kappa\kappa''+8(\kappa')^2}{192}\,ds
	=\frac{1}{192}\int_K\kappa^4ds
\end{gather}
by \eqref{eq:residue-formula}.
The residues at the other poles can be computed similarly.

\begin{dfn}
	We call the number $E(K)=B_K(-2)$ the \emph{energy} of the Legendrian knot $K$.
\end{dfn}

\section{The energy and its M\"obius invariance}
\label{sec:Mobius-invariance}

In this section, we prove the M\"obius invariance of our energy, i.e., the invariance under the action of $\PU(2,1)$,
stated below.

\begin{thm}
	\label{thm:PU-invariance-of-energy}
	Let $K$ be a Legendrian knot in $\mathcal{H}$.
	Then
	\begin{equation}
		E(K)=E(T(K))
	\end{equation}
	for any $T\in\PU(2,1)$ satisfying $T(K)\subset\mathcal{H}$.
\end{thm}

Observe that it is immediate from the definition that $E(K)=B_K(-2)$ is invariant under
the dilations, the left translations, and the rotations.
Thus it remains to verify the invariance under the inversion.

Let us briefly recall how the corresponding result was established
for the Euclidean knot energy by Freedman--He--Wang \cite{Freedman-He-Wang-94}.
In this setting, the energy was originally defined by the second author \cite{OHara-91} by the formula
\begin{equation}
	\label{eq:OHara-energy-for-Euclidean-knots}
	E(K)
	=\lim_{\varepsilon\to +0}
	\left(\iint_{(K\times K)\setminus\Delta_\varepsilon}\frac{dp\,dq}{\abs{p-q}^2}
	-\frac{2L}{\varepsilon}\right)
\end{equation}
using the method of Hadamard regularization,
where $\Delta_\varepsilon=\set{(p,q)\in K\times K|\abs{p-q}<\varepsilon}$
and $L$ is the length of $K$.
Afterwards, Nakauchi \cite{Nakauchi-93} and \cite{Freedman-He-Wang-94} slightly reformulated it,
by using the arc-length distance $d_K$, as
\begin{equation}
	\label{eq:NFHW-formulation-of-energy-for-Euclidean-knots}
	E(K)
	=-4+\iint_{K\times K}\left(\frac{1}{\abs{p-q}^2}-\frac{1}{d_K(p,q)^2}\right)dp\,dq,
\end{equation}
and this expression was used for showing the M\"obius invariance of $E(K)$.
Finally, the equivalence to the definition based on the beta function was observed
by Brylinski \cite{Brylinski-99}.

Along the same lines, we can reformulate our energy of Legendrian knots in $\mathcal{H}$.
Let us introduce the pointwise \emph{potential} at $p\in K$ by
\begin{equation}
	\label{eq:OHara-type-potential-function}
	V(K;p)
	=\lim_{\varepsilon\to+0}
	\left(\int_{K\setminus\BK(p;\varepsilon)}\frac{dq}{\dK(p,q)^2}-\frac{2}{\varepsilon}\right),
\end{equation}
which converges by the asymptotic behavior \eqref{eq:asymptotics-of-signed-chord-length} of
the signed chord length function.
Then we can show that
\begin{equation}
	\label{eq:OHara-type-formulation-of-Legendrian-knot-energy}
	\begin{split}
	E(K)
	=\int_KV(K;p)\,dp
	&=\lim_{\varepsilon\to +0}
	\left(\iint_{\dK(p,q)\geqq\varepsilon}\frac{dp\,dq}{\dK(p,q)^2}
	-\frac{2L}{\varepsilon}\right) \\
	&=
	-4+\iint_{K\times K}\left(\frac{1}{\dK(p,q)^2}-\frac{1}{d_K(p,q)^2}\right)dp\,dq
	\end{split}
\end{equation}
following the argument in \cite{Brylinski-99}, where $d_K$ is the arc-length distance.
Since $\dK(p,q)\leqq d_K(p,q)$ by Corollary \ref{cor:length-in-terms-of-Koranyi-distance},
it follows that $E(K)$ is bounded from below by $-4$;
however, the actual minimum energy is $0$, as we discuss in the next section.

\begin{rem}
	Our energy $E$ is an \emph{energy of knots} in the sense of \cite{OHara-03}.
	In fact, we have just observed the boundedness from below,
	and the continuity with respect to the $C^2$-topology can be proved in the same way as in
	the Euclidean case.
	To show the self-repulsiveness, i.e., that $E(K)$ diverges to $+\infty$ as $K$ approaches a self-intersection,
	suppose that there exist $p_0$, $q_0\in K$, $p_0\not=q_0$ such that
	$\dK(p_0,q_0)=\sigma$ and $d_K(p_0,q_0)=\delta$, where $\sigma\ll\delta$.
	Let $\gamma\colon\mathbb{R}\to\mathcal{H}$, with period $L=L(K)$, be the arc-length parametrization of $K$
	for which $\gamma(0)=p_0$ and $\gamma(\delta)=q_0$.
	Then, for any $p=\gamma(s)\in K$ with $\abs{s}\leqq\delta/2$, we have
	\begin{equation}
		\begin{split}
			\int_K\left(\frac{1}{\dK(p,q)^2}-\frac{1}{d_K(p,q)^2}\right)dq
			&\geqq\int_{-\delta/4}^{\delta/4}
			\left(\frac{1}{\dK(p,\gamma(\delta+s'))^2}-\frac{1}{(\delta/4)^2}\right)ds' \\
			&\geqq\int_{-\delta/4}^{\delta/4}\frac{ds'}{(\sigma+\abs{s}+\abs{s'})^2}-\frac{8}{\delta}
			\geqq\frac{2}{\sigma+\abs{s}}-\frac{16}{\delta}.
		\end{split}
	\end{equation}
	Consequently,
	it follows from the second line of \eqref{eq:OHara-type-formulation-of-Legendrian-knot-energy} that
	\begin{equation}
		E(K)
		\geqq -4+\int_{-\delta/2}^{\delta/2}\left(\frac{2}{\sigma+\abs{s}}-\frac{16}{\delta}\right)ds
		\geqq 4\log\delta-4\log\sigma+(\text{constant}).
	\end{equation}
	This proves the desired result.
\end{rem}

In \eqref{eq:OHara-type-potential-function}, one is also allowed to move toward $p$
from the both sides at different speeds.
For simplicity, let us take an arc-length parametrization $\gamma\colon I\to\mathcal{H}$ of $K$
so that $\gamma(0)=p$ and $0$ is in the interior of $I$. Then
\begin{equation}
	\label{eq:OHara-type-potential-function-another-form}
	V(K;p)
	=\lim_{\varepsilon_1,\,\varepsilon_2\to+0}
	\left(\int_{I\setminus(-\delta_1,\delta_2)}\frac{ds}{\dK(p,\gamma(s))^2}
	-\frac{1}{\varepsilon_1}-\frac{1}{\varepsilon_2}\right),
\end{equation}
where
\begin{equation}
	\ch(-\delta_1)=-\varepsilon_1,\qquad
	\ch(\delta_2)=\varepsilon_2.
\end{equation}
Furthermore, as follows from \eqref{eq:asymptotics-of-signed-chord-length},
the limit in \eqref{eq:OHara-type-potential-function-another-form} remains unchanged if we replace it with
\begin{equation}
	\lim_{\varepsilon_1,\,\varepsilon_2\to+0}
	\left(\int_{I\setminus(-\delta_1,\delta_2)}\frac{ds}{\dK(p,\gamma(s))^2}
	-\frac{1}{\delta_1}-\frac{1}{\delta_2}\right).
\end{equation}

\begin{prop}
	\label{prop:inversion-formula-for-pointwise-potential}
	Suppose that $K$ is a Legendrian knot
	not passing through the origin in $\mathcal{H}$, and let $p\in K$.
	Then, if we set $\tilde{K}=\iota(K)$ and $\tilde{p}=\iota(p)$
	for the inversion $\iota$ given by \eqref{eq:inversion},
	\begin{equation}
		V(\tilde{K};\tilde{p})=\norm{p}_{\mathcal{H}}^2V(K;p).
	\end{equation}
\end{prop}

\begin{proof}
	Let $\gamma\colon I\to\mathcal{H}$ and $\tilde{\gamma}\colon\tilde{I}\to\mathcal{H}$ be
	the arc-length parametrizations of $K$ and $\tilde{K}$, respectively,
	with independent variables $s$ and $\tilde{s}$.
	Without loss of generality, we assume that $0$ lies in the interior of both $I$ and $\tilde{I}$,
	that $\gamma(0)=p$ and $\tilde{\gamma}(0)=\tilde{p}$,
	and that the one-to-one smooth mapping $s\mapsto\tilde{s}=\varphi(s)$,
	which describes the correspondence between $s$ and $\tilde{s}$
	near $s=0$ and $\tilde{s}=0$, is strictly increasing.
	By \eqref{eq:inversion-formula-for-horizontal-vector-norm}, we have
	\begin{equation}
		\label{eq:derivative-of-arc-length-reparametrization-for-inversion}
		\varphi'(s)=\frac{1}{\normH{\gamma(s)}^2}.
	\end{equation}

	Let $\chtilde$ be the signed chord length function for $\tilde{K}$ based at $\tilde{p}$,
	which is a one-to-one mapping near $0\in\tilde{I}$.
	For sufficiently small $\tilde{\varepsilon}>0$,
	let $\sigma_\pm=\varphi^{-1}(\chtilde^{-1}(\pm\tilde{\varepsilon}))$.
	It follows from \eqref{eq:derivative-of-arc-length-reparametrization-for-inversion} and
	\eqref{eq:inversion-formula-for-Koranyi-distance} that
	\begin{equation}
		\begin{split}
			V(\tilde{K};\tilde{p})
			&=\lim_{\tilde{\varepsilon}\to+0}
			\left(
				\int_{\dK(\tilde{p},\tilde{q})\geqq\tilde{\varepsilon}}
				\frac{d\tilde{q}}{\dK(\tilde{p},\tilde{q})^2}
				-\frac{2}{\tilde{\varepsilon}}
			\right)
			=\lim_{\tilde{\varepsilon}\to+0}
			\left(
				\int_{\dK(\tilde{p},\tilde{\gamma}(\tilde{s}))\geqq\tilde{\varepsilon}}
				\frac{d\tilde{s}}{\dK(\tilde{p},\tilde{\gamma}(\tilde{s}))^2}
				-\frac{2}{\tilde{\varepsilon}}
			\right)\\
			&=\lim_{\tilde{\varepsilon}\to+0}
			\left(
				\int_{I\setminus(\sigma_-,\sigma_+)}
				\frac{1}{\dK(\tilde{p},\tilde{\gamma}(\tilde{s}))^2}
				\cdot\frac{ds}{\norm{\gamma(s)}_\mathcal{H}^2}
				-\frac{2}{\tilde{\varepsilon}}
			\right)\\
			&=\norm{p}_\mathcal{H}^2\lim_{\tilde{\varepsilon}\to+0}
			\left(
				\int_{I\setminus(\sigma_-,\sigma_+)}
				\frac{ds}{\dK(p,\gamma(s))^2}-\frac{2}{\norm{p}_\mathcal{H}^2\tilde{\varepsilon}}
			\right).
		\end{split}
	\end{equation}
	On the other hand, by \eqref{eq:OHara-type-potential-function-another-form} and the remark after that,
	\begin{equation}
		V(K;p)
		=\lim_{\delta_1,\,\delta_2\to+0}
		\left(
			\int_{I\setminus(-\delta_1,\delta_2)}
			\frac{ds}{\dK(p,\gamma(s))^2}-\frac{1}{\delta_1}-\frac{1}{\delta_2}
		\right).
	\end{equation}

	We compare the two limits above after setting $-\delta_1=\sigma_-$ and $\delta_2=\sigma_+$.
	Then, since
	\begin{equation}
		\left.\frac{d\sigma_\pm}{d\tilde{\varepsilon}}\right|_{\tilde{\varepsilon}=0}
		=\left.\frac{d(\chtilde^{-1}(\pm\tilde{\varepsilon}))}{d\tilde{\varepsilon}}\right|_{\tilde{\varepsilon}=0}\cdot\varphi'(0)^{-1}
		=\pm\norm{p}_\mathcal{H}^2,
	\end{equation}
	we have the asymptotic expansions
	\begin{equation}
		\sigma_\pm\sim \pm\norm{p}_\mathcal{H}^2\tilde{\varepsilon}+A\tilde{\varepsilon}^2+\dotsb
		\qquad\text{as $\tilde{\varepsilon}\to 0$},
	\end{equation}
	where $A$ is some constant. Hence we obtain
	\begin{equation}
		V(\tilde{K};\tilde{p})-\norm{p}_\mathcal{H}^2V(K;p)
		=\lim_{\tilde{\varepsilon}\to+0}
		\left(\frac{1}{-\sigma_-}+\frac{1}{\sigma_+}-\frac{2}{\norm{p}_\mathcal{H}^2\tilde{\varepsilon}}\right)
		=0,
	\end{equation}
	which completes the proof.
\end{proof}

Theorem \ref{thm:PU-invariance-of-energy} follows immediately from
Proposition \ref{prop:inversion-formula-for-pointwise-potential}.

\begin{proof}[Proof of Theorem \ref{thm:PU-invariance-of-energy}]
	It suffices to consider the case where $T$ is the inversion.
	Let $\tilde{K}=\iota(K)$, and parametrize $K$ and $\tilde{K}$ in the same way as in the proof of
	Proposition \ref{prop:inversion-formula-for-pointwise-potential}.
	Then,
	\begin{equation}
		\begin{split}
			E(\tilde{K})
			&=\int_{\tilde{I}}V(\tilde{K};\tilde{\gamma}(\tilde{s}))\,d\tilde{s}\\
			&=\int_IV(\tilde{K};\tilde{\gamma}(\tilde{s}))\,\frac{ds}{\norm{\gamma(s)}_\mathcal{H}^2}
			=\int_I\norm{\gamma(s)}_\mathcal{H}^2V(K;\gamma(s))\,\frac{ds}{\norm{\gamma(s)}_\mathcal{H}^2}
			=E(K).\qedhere
		\end{split}
	\end{equation}
\end{proof}

\section{$\mathbb{R}$-circles as the least-energy Legendrian knots}
\label{sec:R-circles}

In this section, we determine those Legendrian knots for which the energy $E(K)$ takes its minimum value.
For this purpose,
it is important to note that the proof of Proposition \ref{prop:inversion-formula-for-pointwise-potential}
works in greater generality.

From now on, we consider smooth Legendrian knots in $\mathcal{H}\cup\set{\infty}\cong S^3$ in general.
Those passing through $\infty$ are referred to as \emph{open} or \emph{infinite Legendrian knots}.
(For comparison, Legendrian knots are called \emph{closed} or \emph{finite} if not open.)
For an open Legendrian knot $K$, we extend the definition of the energy by setting
\begin{equation}
	E(K)=\int_KV(K;p)\,dp,
\end{equation}
where we continue to define $V(K;\orddot)$ by formula \eqref{eq:OHara-type-potential-function}.
We must be warned that
the second expression in \eqref{eq:OHara-type-formulation-of-Legendrian-knot-energy} of the energy
does not hold in the same form.
We can check that
\begin{equation}
	\label{eq:alternative-energy-formula-for-open-Legendrian-knot}
	E(K)
	=\iint_{K\times K}\left(\frac{1}{\dK(p,q)^2}-\frac{1}{d_K(p,q)^2}\right)\,dp\,dq
\end{equation}
holds instead for open knots $K$.

Suppose that $K$ is an arbitrary Legendrian knot in $\mathcal{H}\cup\set{\infty}$,
not necessarily avoiding the origin, not necessarily finite.
In this case, $\tilde{K}=\iota(K)$ may also be an open Legendrian knot.
The proof of Proposition \ref{prop:inversion-formula-for-pointwise-potential} still works
in this setting, and we obtain $E(K)=E(\tilde{K})$.
Therefore, the M\"obius invariance of the energy can be reformulated in the following simple form,
for the invariance under dilations, left translations, and rotations is again immediate.

\begin{thm}
	For an arbitrary Legendrian knot $K$ in $\mathcal{H}\cup\set{\infty}$
	and any M\"obius transformation $T$,
	$E(K)=E(T(K))$ holds.
\end{thm}

Next, note that any Legendrian knot $K$, closed or open,
can be mapped to a Legendrian knot $K'$ passing through the origin by a left translation,
without changing its energy.
Let $\tilde{K}'$ be the inversion of $K'$.
Then $\tilde{K}'$ is an open Legendrian knot, and by \eqref{eq:alternative-energy-formula-for-open-Legendrian-knot}
\begin{equation}
	\label{eq:energy-equality-for-opening-knots}
	\begin{split}
		E(K)=E(K')=E(\tilde{K}')
		&=\int_{\tilde{K}'}V(\tilde{K}';\tilde{p})\,d\tilde{p}\\
		&=\iint_{\tilde{K}'\times\tilde{K}'}\left(\frac{1}{\dK(p,q)^2}-\frac{1}{d_{\tilde{K}'}(p,q)^2}\right)\,dp\,dq.
	\end{split}
\end{equation}
This shows that the energy is always non-negative, and vanishes
if and only if $\dK(p,q)=d_{\tilde{K}'}(p,q)$ everywhere along $\tilde{K}'$.

Our claim is the following: the last condition implies that $\tilde{K}'$ is an infinite $\mathbb{R}$-circle.

Before proceeding, let us briefly recall the definition of $\mathbb{R}$-circles
by following Goldman \cite{Goldman-99}.
To define them in an abstract manner, it is convenient to equip the complex unit ball $B^2_\mathbb{C}$
with the complex hyperbolic metric.
Then, it is known that any complete totally geodesic surface in $B^2_\mathbb{C}$ is either a complex submanifold
or a totally real submanifold (and every such surface is isometric to the hyperbolic plane).
A knot in $\bdry B^2_\mathbb{C}$ is called an \emph{$\mathbb{R}$-circle} if
it is the intersection of $\bdry B^2_\mathbb{C}$ and the closure of
a totally geodesic totally real surface in $B^2_\mathbb{C}$.
Any $\mathbb{R}$-circle is automatically Legendrian.
The notion of $\mathbb{R}$-circles is M\"obius invariant
because the action of the group $\PU(2,1)$ on $\bdry B^2_\mathbb{C}$
continuously extends to an action by isometries on $B^2_\mathbb{C}$.
Moreover, $\PU(2,1)$ acts transitively on the set of all $\mathbb{R}$-circles.

Since $\bdry B^2_\mathbb{C}$ can be identified with $\mathcal{H}\cup\set{\infty}$,
$\mathbb{R}$-circles in $\mathcal{H}\cup\set{\infty}$ are also defined.
An $\mathbb{R}$-circle $K$ is called a \emph{finite $\mathbb{R}$-circle} if $K\subset\mathcal{H}$;
otherwise $K$ is an \emph{infinite $\mathbb{R}$-circle}.

Infinite $\mathbb{R}$-circles are easy to describe \cite{Goldman-99}*{Corollary 4.4.4}:
an infinite Legendrian knot $K$ is an infinite $\mathbb{R}$-circle if and only if
its vertical projection is an affine line in $\mathbb{C}$.
Therefore, more explicitly, infinite $\mathbb{R}$-circles are affine lines in $\mathcal{H}$ given by
\begin{equation}
	\Set{(x_0,y_0,u_0)+t\left(a,b,-\frac{1}{2}(y_0a-x_0b)\right)|t\in\mathbb{R}},\qquad
	\text{where $(a,b)\in\mathbb{R}^2\setminus\set{(0,0)}$}.
\end{equation}
In particular, infinite $\mathbb{R}$-circles passing through the origin in $\mathcal{H}$ are
\begin{equation}
	\label{eq:standard-infinite-R-circle}
	\Gamma_{(a,b)}=\Set{t(a,b,0)|t\in\mathbb{R}},\qquad
	\text{where $(a,b)\in\mathbb{R}^2\setminus\set{(0,0)}$}.
\end{equation}
Since $\PU(2,1)$ acts transitively on the set of all $\mathbb{R}$-circles,
any finite $\mathbb{R}$-circle is the image of an (in fact, any) infinite $\mathbb{R}$-circle
under a M\"obius transformation.
Thus finite $\mathbb{R}$-circles are also completely understood in this sense,
although they are generally more difficult to work with.

\begin{prop}
	\label{prop:no-hanging-around-means-infinite-R-circle}
	Let $K$ be an open Legendrian knot for which
	\begin{equation}
		\dK(p,q)=d_K(p,q)\qquad\text{for any $p$, $q\in K\setminus\set{\infty}$}.
	\end{equation}
	Then, $K$ is an infinite $\mathbb{R}$-circle.
\end{prop}

\begin{proof}
	It suffices to show that the vertical projection of $K$ has
	vanishing curvature everywhere.
	Let $p\in K\setminus\set{\infty}$ be arbitrary,
	and we use an arc-length parametrization of $K$ in which $p$ corresponds to $s=0$.
	We apply left translations and rotations to $K$ so that $p$ becomes the origin
	and the velocity vector there equals $(1,0,0)$ (these transformations
	affect $\pi(K)$ only by Euclidean isometries, so the curvature remains unchanged).
	Then, if $\kappa$ is the signed curvature of $\pi(K)$,
	$K$ is parametrized as in \eqref{eq:x-coord-of-legendrian-curves}, \eqref{eq:y-coord-of-legendrian-curves},
	and \eqref{eq:u-coord-of-legendrian-curves}.
	This implies that
	\begin{equation}
		\dK(0,\gamma(s))
		=\sqrt[4]{(x(s)^2+y(s)^2)^2+16u(s)^2}
		\sim s-\frac{\kappa^2}{72}s^3+\dotsb.
	\end{equation}
	However, the assumption implies that the left-hand side equals $s$,
	which means that $\kappa$ must be zero at $p$.
\end{proof}

\begin{thm}
	\label{thm:minimum-energy-knot}
	The energy $E(K)$ of Legendrian knots attains its minimum value $0$
	exactly at $\mathbb{R}$-circles.
\end{thm}

\begin{proof}
	If $E(K)=0$, then \eqref{eq:energy-equality-for-opening-knots}
	and Proposition \ref{prop:no-hanging-around-means-infinite-R-circle}
	implies that $K$ is the image of an infinite $\mathbb{R}$-circle
	under a M\"obius transformation, which is again an $\mathbb{R}$-circle.
\end{proof}

\section{The cosine formula of the energy}
\label{sec:cosine-formula}

For ordinary knots in $\mathbb{R}^3$, Doyle and Schramm proved the formula
\begin{equation}
	\label{eq:classical-cosine-formula}
	E(K)=\iint_{K\times K}\frac{1-\cos\theta_K(p,q)}{\abs{p-q}^2}\,dp\,dq,
\end{equation}
where $\theta_K(p,q)$ is a certain M\"obius invariant angle ``between $p$ and $q$ relative to $K$.''
This is called the \emph{cosine formula}. Our plan here is to show its Heisenberg analog.

Our version of $\theta_K(p,q)$ for Legendrian knots $K$ is defined as follows
in terms of $\mathbb{R}$-circles introduced in the last section.
Recall that the classification of infinite $\mathbb{R}$-circles we discussed earlier immediately implies
that, for any distinct points $p$, $q\in\mathcal{H}\cup\set{\infty}$
and a horizontal direction $l$ at $p$,
there exists a unique $\mathbb{R}$-circle passing through $p$ and $q$ that is tangent to $l$ at $p$
\cite{Goldman-99}*{Theorem 4.4.12}.

\begin{dfn}
	Let $K$ be a Legendrian knot, possibly infinite.
	For any two distinct points $p$, $q\in K$,
	let $\Gamma_K(p,q)$ denote the $\mathbb{R}$-circle passing through $p$ and $q$ that is tangent to $K$ at $p$.
	Then, the angle $\theta_K(p,q)$ is defined to be the angle
	between the two $\mathbb{R}$-circles $\Gamma_K(p,q)$ and $\Gamma_K(q,p)$,
	measured either at $p$ or at $q$.
	More precisely, we equip $K$ with an arbitrary orientation, which induces
	orientations on $\Gamma_K(p,q)$ and $\Gamma_K(q,p)$,
	and the angle is measured with respect to these orientations.
	Thus $\theta_K(p,q)$ is a well-defined angle taking values in $[0,\pi]$.
\end{dfn}

The angle between two $\mathbb{R}$-circles (or Legendrian curves in general)
is a M\"obius invariant notion since, as is well known, M\"obius transformations preserve the conformal class of
the metric defined in the contact distribution.
The equality of the angles measured at $p$ and at $q$ can be seen as follows.
By a M\"obius transformation, we may assume that $p=(0,0,0)$ and $q=\infty$.
Then $\Gamma_\gamma(p,q)=\Gamma_{(a,b)}$ and $\Gamma_\gamma(q,p)=\Gamma_{(a',b')}$
for some $(a,b)$ and $(a',b')$,
where $\Gamma_{(a,b)}$ denotes the infinite $\mathbb{R}$-circle given by \eqref{eq:standard-infinite-R-circle}.
The angle between these two $\mathbb{R}$-circles at $p=(0,0,0)$
is the angle given by two directions $(a,b)$ and $(a',b')$.
To measure the angle at $q=\infty$, we apply the inversion, which sends $q$ to the origin.
Since both $\Gamma_{(a,b)}$ and $\Gamma_{(a',b')}$ are invariant under the inversion,
the angle at $q$ is again the angle between the directions $(a,b)$ and $(a',b')$.

We claim the following.

\begin{thm}
	\label{thm:cosine-formula}
	For any finite Legendrian knot $K$,
	\begin{equation}
		E(K)
		=\iint_{K\times K}\frac{1-\cos\theta_K(p,q)}{\dK(p,q)^2}\,dp\,dq
	\end{equation}
	holds.
\end{thm}

\begin{proof}
	It suffices to show that
	\begin{equation}
		\label{eq:pointwise-cosine-formula}
		V(K;p)
		=\int_K\frac{1-\cos\theta_K(p,q)}{\dK(p,q)^2}\,dq
	\end{equation}
	for each fixed $p\in K$.
	To prove \eqref{eq:pointwise-cosine-formula},
	a similar argument to the Euclidean case works if we pass to the vertical projection,
	as we elaborate below.
	Notice that the vertical projection $\pi\colon\mathcal{H}\to\mathbb{C}$
	preserves not only the lengths of Legendrian curves
	but also the angle between horizontal tangent vectors.

	By applying a left translation and a rotation to $K$,
	we may assume that $p=(0,0,0)$ and that $K$ is tangent to the vector $(1,0,0)$ at $p$.
	Let $\tilde{K}$ be the inversion of $K$, which is an open Legendrian knot.
	Then, one can check that an infinite $\mathbb{R}$-circle $\Gamma$
	is tangent to $\tilde{K}$ at $\tilde{p}=\infty$ if and only if
	$\pi(\Gamma)$ is an affine line in $\mathbb{C}$ parallel to the vector $(1,0)$.

	Let us parametrize $K$ by the arc-length parameter $s$ so that $p=(0,0,0)$ corresponds to $s=0$.
	Let $\ch$ be the signed chord length function defined by \eqref{eq:signed-chord-length},
	and for sufficiently small $\varepsilon>0$,
	we define $p_\pm\in K$ to be the points corresponding to $s=\ch^{-1}(\pm\varepsilon)$.
	Then,
	\begin{equation}
		V(K;p)
		=\lim_{\varepsilon\to+0}\left(\int_{p_+}^{p_-}\frac{dq}{\dK(p,q)^2}-\frac{2}{\varepsilon}\right)
		=\lim_{\varepsilon\to+0}\left(\int_{\tilde{p}_+}^{\tilde{p}_-}d\tilde{q}-\frac{2}{\varepsilon}\right)
	\end{equation}
	by \eqref{eq:inversion-formula-for-horizontal-vector-norm},
	where $\tilde{p}_\pm=\iota(p_\pm)$.

	To identify the latter limit in terms of $\cos\theta_K$, we discuss the asymptotic behavior of
	$\tilde{p}_\pm$ as $\varepsilon\to+0$.
	Recall that $p=(x(s),y(s),u(s))\in K$ can be asymptotically written down in terms of
	the signed curvature $\kappa$ of the vertical projection $\pi(K)$
	as \eqref{eq:x-coord-of-legendrian-curves}, \eqref{eq:y-coord-of-legendrian-curves},
	and \eqref{eq:u-coord-of-legendrian-curves}.
	Some computation shows, therefore,
	that $\tilde{p}=\iota(p)=(\tilde{x}(s),\tilde{y}(s),\tilde{u}(s))$ can be given by
	\begin{equation}
		\tilde{x}(s)=\frac{1}{s}+O(s),\qquad
		\tilde{y}(s)=\frac{\kappa}{6}+O(s),\qquad
		\tilde{u}(s)=-\frac{\kappa}{12}\cdot\frac{1}{s}+\frac{\kappa'}{24}+O(s).
	\end{equation}
	Consequently, in view of \eqref{eq:asymptotics-of-signed-chord-length-inverse}, we have
	\begin{equation}
		\label{eq:asymptotic-behavior-of-infinite-Legendrian-knot}
		\pi(\tilde{p}_\pm)
		=\left(\pm\frac{1}{\varepsilon}+O(\varepsilon),\frac{\kappa}{6}+O(\varepsilon)\right)
		\qquad\text{as $\varepsilon\to+0$}.
	\end{equation}
	By definition, $\theta_{\tilde{K}}(\tilde{p},\tilde{q})$ is the angle between
	$\Gamma_{\tilde{K}}(\infty,\tilde{q})$,
	which is one of the infinite $\mathbb{R}$-circles tangent to $\tilde{K}$ at $\infty$,
	and $\Gamma_{\tilde{K}}(\tilde{q},\infty)$,
	the infinite $\mathbb{R}$-circle tangent to $\tilde{K}$ at $\tilde{q}$.
	The angle they form actually equals the angle in $\mathbb{C}$ formed by
	$\pi(\Gamma_{\tilde{K}}(\infty,\tilde{q}))$ and $\pi(\Gamma_{\tilde{K}}(\tilde{q},\infty))$,
	or equivalently, the angle between the vector $(1,0)$ and
	the tangent vector of $\pi(\tilde{K})$ at $\pi(\tilde{q})$.
	Therefore, by \eqref{eq:asymptotic-behavior-of-infinite-Legendrian-knot},
	\begin{equation}
		\lim_{\varepsilon\to+0}\left(\int_{\tilde{p}_+}^{\tilde{p}_-}\cos\theta_{\tilde{K}}(\tilde{p},\tilde{q})\,d\tilde{q}-\frac{2}{\varepsilon}\right)=0,
	\end{equation}
	and hence we obtain
	\begin{equation}
		V(K;p)
		=\lim_{\varepsilon\to+0}\int_{\tilde{p}_+}^{\tilde{p}_-}(1-\cos\theta_{\tilde{K}}(\tilde{p},\tilde{q}))\,d\tilde{q}
		=\int_K\frac{1-\cos\theta_K(p,q)}{\dK(p,q)^2}\,dq
	\end{equation}
	by using \eqref{eq:inversion-formula-for-horizontal-vector-norm} again.
	This completes the proof.
\end{proof}

Theorem \ref{thm:cosine-formula} recovers Theorem \ref{thm:PU-invariance-of-energy} as a corollary
because the 2-form \eqref{eq:invariant-2-form-on-KxK} is M\"obius invariant.
It also recovers Theorem \ref{thm:minimum-energy-knot} because of the following proposition.

\begin{prop}
	Let $K$ be a Legendrian knot such that
	$\theta_K(p,q)$ is either $0$ or $\pi$ for any $p$, $q\in K$ with $p\not=q$.
	Then $K$ is an $\mathbb{R}$-circle (and hence $\theta_K(p,q)=0$ for any $p$, $q$).
\end{prop}

\begin{proof}
	Take any $p_0\in K$.
	For each $q\in\mathcal{H}\setminus\set{p_0}$,
	let $L_{p_0,v_0}(q)\subset H_q$ denote the tangent line at $q$ of $\Gamma_K(p_0,q)$.
	Then $L_{p_0,v_0}$ defines a smooth line field on $\mathcal{H}\setminus\set{p_0}$.
	By the assumption, $K$ is an integral curve of this line field.
	On the other hand, for any fixed $q_0\in K\setminus\set{p_0}$,
	$\Gamma_0=\Gamma_K(p_0,q_0)$ is also an integral curve.
	The uniqueness of solutions to ODEs implies that $K$ and $\Gamma_0$ coincide.
\end{proof}

\section{Geometric interpretations of the integrand}
\label{sec:interpretations-of-integrand}

In the classical case of knots in $\mathbb{R}^3$, Langevin--O'Hara \cite{Langevin-OHara-05} pointed out that
the 2-form $dp\,dq/\abs{p-q}^2$ on $(K\times K)\setminus\Delta$, the integrand for the energy,
can be interpreted as the absolute value of the \emph{infinitesimal cross ratio}.
The infinitesimal cross ratio $\Omega$ is the complex 2-form
defined as the cross ratio of four points $p$, $p+dp$, $q$, $q+dq$
under the identification
of the 2-sphere (possibly of infinite radius) passing through these four points and $\mathbb{C}$ via
the stereographic projection.
More precisely, the 2-sphere can be equipped with a canonical orientation,
which ensures the well-definedness of $\Omega$.

In fact, $\Omega$ can be defined as
a complex 2-form on $(\mathbb{R}^3\times\mathbb{R}^3)\setminus\Delta$
or on $(S^3\times S^3)\setminus\Delta$, where $S^3$ is the one-point compactification of $\mathbb{R}^3$,
and the infinitesimal cross ratio for each individual knot $K$ is then the pullback of $\Omega$ to
$(K\times K)\setminus\Delta$.
From this perspective, one can also observe that the real part of $\Omega$ equals the canonical symplectic form
of the cotangent bundle $T^*S^3$
via a natural identification $(S^3\times S^3)\setminus\Delta\cong T^*S^3$
(see \cite{Langevin-OHara-05}*{Lemma 4.16}),
and the angle relative to $K$ is nothing but the absolute value of the argument of $\Omega$.
Therefore, the classical cosine formula \eqref{eq:classical-cosine-formula} can be rewritten
purely in terms of $\Omega$ as
\begin{equation}
	E(K)=\iint_{K\times K}\abs{\Omega}-\Re\Omega.
\end{equation}
This expression can be used to recover the original definition \eqref{eq:OHara-energy-for-Euclidean-knots}
of the energy,
thereby closing the circle of arguments.

In what follows, we discuss the integrand $dp\,dq/\dK(p,q)^2$ for our energy for Legendrian knots
from similar viewpoints.

Following Kor\'anyi--Reimann \cite{Koranyi-Reimann-87},
for $p=(z,u)\in\mathcal{H}$, let
\begin{equation}
	A(p)=\abs{z}^2-4iu
\end{equation}
(the factor $-4$ is due to our normalization for the group structure on $\mathcal{H}$).
Note that
\begin{equation}
	\iota(p)=\left(\frac{z}{\conj{A(p)}},-\frac{u}{\abs{A(p)}^2}\right)
	\qquad\text{for $p=(z,u)$}.
\end{equation}
We also have $\normH{p}=\abs{A(p)}^{1/2}$ and hence $\dK(p,q)=\abs{A(p^{-1}q)}^{1/2}$.
Furthermore, $A(p^{-1})=\conj{A(p)}$, $A(\iota(p))=1/A(p)$, and
\begin{equation}
	\label{eq:A-pinv-q}
	A(p^{-1}q)=\conj{A(p)}+A(q)-2z\conj{z}'
	\qquad\text{for $p=(z,u)$ and $q=(z',u')$}.
\end{equation}

The \emph{complex cross ratio} as defined in \cite{Koranyi-Reimann-87} is given by
\begin{equation}
	\mathbb{X}(p_1,p_2,p_3,p_4)
	=\frac{A(p_3^{-1}p_1)A(p_4^{-1}p_2)}{A(p_4^{-1}p_1)A(p_3^{-1}p_2)}.
\end{equation}
It is known to be M\"obius invariant, i.e., $\mathbb{X}(T(p_1),T(p_2),T(p_3),T(p_4))=\mathbb{X}(p_1,p_2,p_3,p_4)$
for any $T\in\PU(2,1)$.
The following proposition shows that the ``infinitesimal complex cross ratio''
gives the square of the integrand for our Heisenberg energy.

\begin{prop}
	\label{prop:infinitesimal-complex-cross-ratio}
	For a Legendrian curve $\gamma$,
	\begin{equation}
		\lim_{\Delta s,\,\Delta t\to 0}
		\frac{\mathbb{X}(\gamma(s),\gamma(t),\gamma(s+\Delta s),\gamma(t+\Delta t))}{(\Delta s)^2(\Delta t)^2}
		=\left(\frac{\abs{\gamma'(s)}\abs{\gamma'(t)}}{\normH{\gamma(s)^{-1}\gamma(t)}^2}\right)^2
	\end{equation}
	holds for $s\not=t$.
\end{prop}

\begin{proof}
	Without loss of generality, we may assume that $\gamma$ is parametrized by arc length.
	We evaluate the numerator and the denominator of
	\begin{equation}
		\conj{\mathbb{X}(\gamma(s),\gamma(t),\gamma(s+\Delta s),\gamma(t+\Delta t))}
		=\frac{A(\gamma(s)^{-1}\gamma(s+\Delta s))A(\gamma(t)^{-1}\gamma(t+\Delta t))}{A(\gamma(s)^{-1}\gamma(t+\Delta t))A(\gamma(t)^{-1}\gamma(s+\Delta s))}.
	\end{equation}
	For the numerator,
	applying the asymptotic expansions \eqref{eq:x-coord-of-legendrian-curves},
	\eqref{eq:y-coord-of-legendrian-curves}, \eqref{eq:u-coord-of-legendrian-curves} for Legendrian curves,
	we obtain
	\begin{equation}
		\gamma(s)^{-1}\gamma(s+\Delta s)
		=(\Delta s+O((\Delta s)^3),O((\Delta s)^2),O((\Delta s)^3))
	\end{equation}
	and hence
	\begin{equation}
		A(\gamma(s)^{-1}\gamma(s+\Delta s))
		=(\Delta s)^2+O((\Delta s)^3).
	\end{equation}
	For the denominator,
	\begin{equation}
		\begin{split}
			A(\gamma(s)^{-1}\gamma(t+\Delta t))A(\gamma(t)^{-1}\gamma(s+\Delta s))
			&=\abs{A(\gamma(s)^{-1}\gamma(t))}^2+O(\Delta s)+O(\Delta t)\\
			&=\normH{\gamma(s)^{-1}\gamma(t)}^4+O(\Delta s)+O(\Delta t).
		\end{split}
	\end{equation}
	Combining these estimates yields the desired result.
\end{proof}

Proposition \ref{prop:infinitesimal-complex-cross-ratio} shows that the
``infinitesimal complex cross ratio'' itself does not play the same role
as $\comptwoform$ in the classical case.
What corresponds to $\Omega$ is the complex 2-form
\begin{equation}
	\label{eq:2-form-on-HxH}
	\comptwoformH=d_pd_q\log\rho,\qquad
	\rho=\rho(p,q)=A(p^{-1}q)
\end{equation}
on $(\mathcal{H}\times\mathcal{H})\setminus\Delta$.
The 2-form $\comptwoformH$ is invariant under the diagonal action of M\"obius transformations.
Actually, the function $\rho$ is already invariant for left-translations and rotations, and for dilations we have
\begin{equation}
	\rho(\delta_\lambda(p),\delta_\lambda(q))=\lambda^2\rho(p,q).
\end{equation}
For the inversion, it follows from \eqref{eq:A-pinv-q} that
\begin{equation}
	\log\rho(\iota(p),\iota(q))=\log\rho(p,q)-\log\conj{A(p)}-\log A(q),
\end{equation}
and hence
\begin{equation}
	(\iota\times\iota)^*\comptwoformH
	=d_pd_q\log\rho(\iota(p),\iota(q))
	=d_pd_q\log\rho(p,q)
	=\comptwoformH.
\end{equation}

\begin{prop}
	\label{prop:complex-2-form}
	Let $(p,q)\in(\mathcal{H}\times\mathcal{H})\setminus\Delta$.
	Then, for any horizontal vectors $v\in H_p$ and $v'\in H_q$, we have
	\begin{equation}
		\label{eq:absolute-value-of-2-form}
		\abs{\comptwoformH(v,v')}=2\cdot\frac{\abs{v}\abs{v'}}{\dK(p,q)^2}.
	\end{equation}
	Moreover, if $K$ is a Legendrian knot passing through $p$, $q$
	and $v$ and $v'$ are tangent to $K$, then
	the absolute value of the argument of $\comptwoformH(v,v')$ equals $\theta_K(p,q)$.
\end{prop}

\begin{proof}
	Let $p=(z,u)$ and $q=(z',u')$ for general $p$ and $q$.
	Then, \eqref{eq:A-pinv-q} implies
	\begin{equation}
		\rho(p,q)=\abs{z}^2+\abs{z'}^2-4i(u'-u)-2z\conj{z}'.
	\end{equation}
	A straightforward computation shows that
	\begin{equation}
		d_p\rho=-2(\conj{z}'-\conj{z})\,dz+4i\theta,\qquad
		d_q\rho=2(z'-z)\,d\conj{z}'-4i\theta',\qquad
		d_pd_q\rho=-2\,dz\wedge d\conj{z}',
	\end{equation}
	where $\theta$ and $\theta'$ denote the left-invariant contact form \eqref{eq:contact-form}
	on each factor of $\mathcal{H}\times\mathcal{H}$.
	Consequently,
	\begin{equation}
		\begin{split}
			\comptwoformH
			=\frac{d_pd_q\rho}{\rho}-\frac{d_p\rho\wedge d_q\rho}{\rho^2}
			\equiv 2\frac{\conj{\rho}}{\rho^2}\,dz\wedge d\conj{z}'
			\mod\theta,\,\theta'.
		\end{split}
	\end{equation}
	This implies \eqref{eq:absolute-value-of-2-form}
	because $\abs{dz(v)}=\abs{v}$, $\abs{d\conj{z}'(v')}=\abs{v'}$,
	and $\abs{\rho}=\abs{A(p^{-1}q)}=\dK(p,q)^2$.

	To show the latter assertion, we use its M\"obius invariance:
	we may assume without losing generality that $p=(0,0,0)$, $q=(1,0,0)$, and $v=(1,0,0)$.
	Then the desired result immediately follows because $dz(v)=1$ and
	$d\conj{z}'(v')=\abs{v'}e^{\pm i\theta_K(p,q)}$.
\end{proof}

By Theorem \ref{thm:cosine-formula} and Proposition \ref{prop:complex-2-form},
we obtain the following corollary.

\begin{cor}
	For any Legendrian knot $K$,
	\begin{equation}
		\label{eq:energy-as-abs-minus-real-part}
		E(K)=\frac{1}{2}\iint_K\abs{\comptwoformH}-\Re\comptwoformH
	\end{equation}
	holds.
\end{cor}

The relationship between the real part of $\comptwoformH$ and the canonical symplectic form
on $T^*\mathcal{H}$ can be established as follows.
Let us consider the 1-form $\alpha$ on $(\mathcal{H}\times\mathcal{H})\setminus\Delta$
given by
\begin{equation}
	\alpha=\Re d_q\log\rho,
\end{equation}
where $\rho$ is defined as in \eqref{eq:2-form-on-HxH}.
Note that, for each fixed $p\in\mathcal{H}$, $\alpha_{(p,\orddot)}$ is a 1-form on $\mathcal{H}\setminus\set{p}$,
and it is actually the pullback of the 1-form $\Re d\log\rho(0,\orddot)$ on $\mathcal{H}\setminus\set{0}$
by the left translation $L_{p^{-1}}$.

\begin{prop}
	\label{prop:comparison-with-symplectic-form}
	(1) The mapping $F\colon(\mathcal{H}\times\mathcal{H})\setminus\Delta\to T^*\mathcal{H}$ defined by
	\begin{equation}
		F(p,q)=(q,\alpha_{(p,q)})
	\end{equation}
	provides a diffeomorphism between
	$(\mathcal{H}\times\mathcal{H})\setminus\Delta$
	and $T^*\mathcal{H}\setminus 0_\mathcal{H}$, where $0_\mathcal{H}$ denotes the zero section.

	(2) The real part of $\comptwoformH$ is equal to the pullback of the canonical symplectic form on
	$T^*\mathcal{H}$ by $F$.
\end{prop}

It should be noted that the definition of $F$ is perhaps less geometrically natural
than the identification $(S^3\times S^3)\setminus\Delta\cong T^*S^3$ in the classical case.

\begin{proof}[Proof of Proposition \ref{prop:comparison-with-symplectic-form}]
	(1) 
	We write $p=(z,u)$ and $q=(z',u')$. Then,
	\begin{equation}
		\Re d\log\rho(0,q)
		=\Re d\log A(q)=\frac{1}{2}d\log\abs{A(q)}^2
		=\frac{\abs{z'}^2(\conj{z}'dz'+z'd\conj{z}')+16u'du'}{\abs{A(q)}^2}
	\end{equation}
	and hence
	\begin{equation}
		\begin{split}
			\alpha_{(p,q)}
			&=(L_{p^{-1}}^*(\Re d\log\rho(0,\orddot)))_q \\
			&=\frac{\abs{z'-z}^2((\conj{z}'-\conj{z})dz'+(z'-z)d\conj{z}')+16(u'-u+\frac{1}{2}\Im(z\conj{z}'))(du'+\frac{1}{2}\Im(z\,d\conj{z}'))}{\abs{A(p^{-1}q)}^2}.
		\end{split}
	\end{equation}
	If we write $r=p^{-1}q=(z'',u'')$, we obtain
	\begin{equation}
		\label{eq:good-expression-of-alpha}
		\begin{split}
			\alpha_{(p,q)}
			&=\frac{\abs{z''}^2(\conj{z}''dz'+z''d\conj{z}')+16u''(\theta'-\frac{1}{2}\Im(z''d\conj{z}'))}{\abs{A(r)}^2} \\
			&=\frac{\conj{z}''}{\conj{A(r)}}\,dz'+\frac{z''}{A(r)}\,d\conj{z}'+\frac{16u''}{\abs{A(r)}^2}\,\theta'.
		\end{split}
	\end{equation}
	Since
	\begin{equation}
		\mathcal{H}\setminus\set{0}\to\mathcal{H}\setminus\set{0},\qquad
		r=(z'',u'')\mapsto\left(\frac{z''}{A(r)},\frac{16u''}{\abs{A(r)}^2}\right)
	\end{equation}
	has a smooth inverse (because it is the composition of $(z'',u'')\mapsto (z'',-u'')$, the inversion,
	and $(z'',u'')\mapsto (z'',16u'')$),
	it follows that $F$ is a diffeomorphism onto $T^*\mathcal{H}\setminus 0_\mathcal{H}$.

	(2) This is obvious because $d\alpha=\Re\comptwoformH$ and
	$\alpha$ is the pullback of the tautological 1-form on $T^*\mathcal{H}$.
\end{proof}

\begin{rem}
	As in the classical Euclidean case,
	we can derive the expression \eqref{eq:OHara-type-formulation-of-Legendrian-knot-energy} of the energy
	based on Hadamard regularization directly from \eqref{eq:energy-as-abs-minus-real-part}.
	Suppose that $K$ is the image of $\gamma$ and is parametrized by the arc length. 
	Put
	\begin{equation}
		D_\varepsilon=(K\times K)\setminus\Delta_\varepsilon,\qquad
		\Delta_\varepsilon=\set{(p,q)\in K\times K|\dK(p,q)<\varepsilon}.
	\end{equation}
	Then
	\begin{equation}
		\begin{split}
			E(K)
			&=\frac{1}{2}\lim_{\varepsilon\to+0}\iint_{D_\varepsilon}\abs{\comptwoformH}-\Re\comptwoformH \\
			&=\lim_{\varepsilon\to+0}\left(\iint_{\dK(p,q)\geqq\varepsilon}\frac{dp\,dq}{\dK(p,q)^2}-\frac{1}{2}\int_{\bdry D_\varepsilon}\alpha\right),
		\end{split}
	\end{equation}
	where $\alpha=\Re d_q\log\rho$ as before. 
	Let $p=(z,u)$, $q=(z',u')$, and
	\begin{equation}
		v=\xi X_q+\eta Y_q=(\xi+i\eta)Z_q+(\xi-i\eta)\conj{Z}_q
	\end{equation}
	be the velocity vector of $\gamma$ at $q$, where $X$, $Y$, and $Z$ are given by
	\eqref{eq:horizontal-vector-fields} and \eqref{eq:complex-horizontal-vector-field}.
	Then, by \eqref{eq:good-expression-of-alpha},
	\begin{equation}
		\begin{split}
			\alpha(v)
			&=2\Re\frac{\rho(p,q)(\conj{z}'-\conj{z})(\xi+i\eta)}{\dK(p,q)^4}\\
			&=2\frac{(\Re\rho(p,q))((x'-x)\xi+(y'-y)\eta)-(\Im\rho(p,q))(-(y'-y)\xi+(x'-x)\eta)}{\dK(p,q)^4}.
		\end{split}
	\end{equation}
	Note that $\bdry D_\varepsilon=\Gamma_+\cup(-\Gamma_-)$,
	where $\Gamma_\pm=\set{(p,q)|\gamma^{-1}(q)-\gamma^{-1}(p)\equiv\pm\varepsilon\ \text{mod.~$L$}}$,
	where $L$ is the length of $K$. 
	For $(p,q)\in\Gamma_+$, \eqref{eq:x-coord-of-legendrian-curves}, \eqref{eq:y-coord-of-legendrian-curves},
	and \eqref{eq:u-coord-of-legendrian-curves} imply
	\begin{align}
		\Re\rho(p,q)&=(x'-x)^2+(y'-y)^2=\varepsilon^2+O(\varepsilon^3), \\
		\Im\rho(p,q)&=-4(u'-u)+2(xy'-x'y)=O(\varepsilon^3)
	\end{align}
	and
	\begin{equation}
		(x'-x)\xi+(y'-y)\eta=\varepsilon+O(\varepsilon^3), \qquad
		-(y'-y)\xi+(x'-x)\eta=O(\varepsilon^2).
	\end{equation}
	These estimates imply
	\begin{equation}
		\frac{1}{2}\int_{\Gamma_+}\alpha=\frac{L}\varepsilon+O(\varepsilon).
	\end{equation}
	The integral along $\Gamma_-$ can be calculated similarly.
	Combining the above, we obtain \eqref{eq:OHara-type-formulation-of-Legendrian-knot-energy}.
\end{rem}

There is also a complex-geometric interpretation of the 2-form $\comptwoformH$.
It bears a resemblance to
\begin{equation}
	\partial\conj{\partial}\log(\Im w-\abs{z}^2),
\end{equation}
which is the K\"ahler form (up to a purely imaginary constant factor) of the complex hyperbolic metric
on the Siegel domain $D=\set{(z,w)\in\mathbb{C}^2|\Im w>\abs{z}^2}$.
In fact, these 2-forms are related concretely as follows.
Identifying $D$ with the diagonal of $D\times D$, we extend the function $\Im w-\abs{z}^2$
holomorphically (resp.\ anti-holomorphically) in the first variable $p$ (resp.\ the second variable $q$)
to get the 2-form
\begin{equation}
	\tilde{\Omega}=\partial_p\conj{\partial}_q\log\tilde{\rho},\qquad
	\tilde{\rho}=\frac{w-\conj{w}'}{2i}-z\conj{z}'
\end{equation}
on $D\times D$. Note that we can also write
\begin{equation}
	\tilde{\Omega}=d_pd_q\log\tilde{\rho}.
\end{equation}
The function $\tilde{\rho}$ continuously extends to $\overline{D}\times\overline{D}$.
Moreover, recall that the boundary of the Siegel domain can be identified with $\mathcal{H}$
via \eqref{eq:identification-of-Heisenberg-and-Siegel-boundary}.
With this in mind, let us write
\begin{equation}
	u=-\frac{1}{4}\Re w,\qquad
	r=\Im w-\abs{z}^2,\qquad
	u'=-\frac{1}{4}\Re w',\qquad
	r'=\Im w'-\abs{z'}^2.
\end{equation}
Then we obtain
\begin{equation}
	\begin{split}
		\tilde{\rho}
		&=-2i(u'-u)+\frac{\abs{z'}^2+r'+\abs{z}^2+r}{2}-z\conj{z}'\\
		&=\frac{1}{2}(\abs{z'-z}^2-4i(u'-u)-2i\Im(z\conj{z}')+r+r').
	\end{split}
\end{equation}
Since $\bdry D\times\bdry D$ is given by $r=r'=0$,
this shows that $\tilde{\rho}|_{\bdry D\times\bdry D}$ vanishes only along the diagonal of $\bdry D\times\bdry D$,
and $\tilde{\rho}$ restricts to $\rho/2$ on the complement of the diagonal of $\bdry D\times\bdry D$.
Consequently, the 2-form $\comptwoformH$ equals the pullback (up to the factor 2) of $\tilde{\Omega}$
to $(\bdry D\times\bdry D)\setminus\Delta$.

\bibliography{myrefs}

\end{document}